\documentclass[11pt]{amsart}
\usepackage{amsmath,amsthm,amsfonts,amssymb,latexsym}
\usepackage[all,2cell,dvips]{xy}

\newtheorem{theorem}{Theorem}[section]
\newtheorem{lemma}[theorem]{Lemma}
\newtheorem{proposition}[theorem]{Proposition}
\newtheorem{corollary}[theorem]{Corollary}

\theoremstyle{definition}

\theoremstyle{remark}

\newcommand{\N}{\mathbb{N}}
\newcommand{\Z}{\mathbb{Z}}
\newcommand{\Q}{\mathbb{Q}}
\newcommand{\R}{\mathbb{R}}

\newcommand{\vc}{\operatorname{vc}}
\newcommand{\VC}{\operatorname{VC}}

\newcommand{\beq}{\begin{equation} \setlength\abovedisplayskip{5pt} 
\setlength\belowdisplayskip{5pt}}
\newcommand{\eeq}{\end{equation}}
\newcommand{\bea}{\begin{eqnarray}}
\newcommand{\eea}{\end{eqnarray}}

\def \<{\langle}
\def \>{\rangle}

\def \((  {(\!(}
\def \)) {)\!)}

\numberwithin{equation}{section}

\allowdisplaybreaks[2]

\begin{document}

\bibliographystyle{plain}
\title{VC-density in Divisible Oriented Abelian Groups and their Pairs}
 
\author{Ebru Nay\.Ir}

\address{Bo\u{g}azi\c{c}i \"{U}niversitesi, Istanbul, Turkey}

\email{ebru.nayir@std.bogazici.edu.tr}
\email{ebrunayirr@gmail.com}

\author{Mel\.Issa Özsahakyan}

\address{Mimar Sinan Fine Arts University, Istanbul, Turkey}

\email{melissa.ozsahakyan@msgsu.edu.tr}



\maketitle
\begin{abstract}
We show that the VC-density in certain theories of oriented abelian groups is at most the size of parameter variables, which yields dp-minimality. We further prove that the VC-density of formulas in pairs of such models is bounded by twice the size of parameter variables. This uniform upper bound is shown to be sharp, and as a consequence, we show that such pairs have dp-rank $2$.
\end{abstract}

\section{Introduction}
In this paper, we study oriented abelian groups and their pairs, with a focus on understanding the combinatorial complexity of definable sets via VC-density. An \textit{oriented abelian group} $A$ is a group equipped with a cyclic ternary relation $\mathcal{O}$, called \textit{orientation}, that is preserved under group operation such that the set 
$\big\{(x,y)\in A^2: x\neq y,\ \mathcal{O}(0,x,y)\big\}$
forms a strict linear order on $A\setminus \{0\}$. We denote this linear order by $x<_0y$. These groups form topological groups, where the basic open sets are \textit{orientation intervals} $(a,b)_\mathcal{O}$, consisting of elements between $a$ and $b$ in the counterclockwise direction. An oriented abelian group $A$ is called \textit{regularly dense} if it satisfies an additional condition: for each prime $p$ and all $a,b\in A$ with $\mathcal{O}(0,a,b)$, there exists an element $c\in A$ such that $\mathcal{O}(a,pc,b)$ and $ic<_0(i+1)c$ for all $i\in \{1,\cdots,p-1\}$. 
As established in \cite{melissa}, certain oriented abelian groups satisfy several tame structural properties, such as $\mathcal{O}$-minimality—a notion analogous to o-minimality for ordered abelian groups. In particular, $\mathcal{O}$-minimality implies NIP. We only recall some useful facts for our purposes in Section~\ref{oriented_groups_section} and for more details, we refer \cite{melissa} to the reader. We then focus on pairs of divisible oriented abelian groups $(A,G)$, where $A$ is a regularly dense oriented abelian group, $G$ is a dense subgroup of $A$, and $A/G$ is infinite. We extend the language of oriented abelian groups $L_{\mathcal{O}}:=\{+,-,0,\mathcal{O}\}$ with a unary predicate $V$, so  $L_{\mathcal{O}}^p:=L_{\mathcal{O}}\cup \{V\}$ is the language of pairs of oriented abelian groups, where $V$ is interpreted as the dense subgroup of the universe. The highlighted result of Section \ref{oriented_groups_section} is as follows.

\begin{theorem}\label{main_thm_qe}
In each of the following cases, the $L_{\mathcal{O}}^p$-theory of pairs $(A,G)$ of divisible oriented abelian groups admits  quantifier elimination:
\begin{enumerate}
    \item $A$ and $G$ are both torsion groups, or
    \item $A$ is torsion and $G$ is torsion-free, or
    \item $A$ and $G$ are both torsion-free.
\end{enumerate}
 \end{theorem}

In Section \ref{main_result_section}, we turn our focus to computing \textit{the VC-density} of formulas in the theories described above. The notion of VC-density is a refinement of \textit{VC-dimension} (Vapnik and Chervonenkis-dimension). Laskowski was the first to give a characterization of NIP theories via VC-dimension, thereby introducing this notion into model theory. Indeed, he proved that a complete first-order theory has NIP if and only if VC-dimension of any formula is finite (See \cite{L}). For further background, we refer \cite{VC-density-I} to the reader and do not go into any other details for the notions VC-dimension and VC-density.

\medskip
A crucial concept for our purposes is that of \textit{uniform definability of types over finite sets with parameters (UDTFS)} (See \cite{Guingona}). In \cite{VC-density-I}, Aschenbrenner et al. introduce a stronger version of this notion, namely the \textit{VC $d$ property}, that ensures a uniform bound on the VC-density of formulas in terms of the size of parameter variables. In particular, VC $1$ property implies dp-minimality and yields a linear bound on VC-density. We give necessary tools in Section \ref{main_result_section}, following \cite{VC-density-I} as our main reference. Using the quantifier elimination results of the theories of interest established in Section \ref{oriented_groups_section}, we actually prove the following results. 

\begin{theorem}\label{main_thm1}
    The $L_{\mathcal{O}}$-theories of divisible oriented abelian groups have the VC $1$ property.  
    \end{theorem}
    
     \begin{theorem}\label{main_thm2}
The $L_{\mathcal{O}}^p$-theories of pairs of divisible oriented abelian groups described in Theorem \ref{main_thm_qe}  have the VC $2$ property.
 \end{theorem}

We also show that the upper bound derived from the second theorem is optimal.  It is already known that divisible oriented abelian groups with torsion, as well as pairs whose universe is divisible and torsion, have NIP, as shown in \cite{melissa}.  
As a by-product, we also deduce that all theories in Theorem \ref{main_thm_qe} have NIP and dp-rank $2$.

\medskip \emph{Notations and Conventions.}
Throughout this paper, an oriented abelian group is always assumed to be regularly dense. For an oriented abelian group $A$, we denote the subgroup of $p$-torsion elements of $A$ by $\operatorname{Tor}_p(A)$, and the subgroup of all torsion elements by $\operatorname{Tor}(A)$.
For any abelian group $A$, we use the notation $[p]A$ for the cardinality of the quotient group $A/pA$.
We denote tuples of variables by $\vec{x}$, and their length by $|\vec{x}|$. When $|\vec{x}|=1$, we write $x$ instead of $\vec{x}$. For a tuple $\vec{y}$ and a tuple $\vec{k}\in \Z^{|\vec{y}|}$, their dot product is written as $\vec{k}\vec{y}$.

\section{Oriented Abelian Groups And Their Pairs}\label{oriented_groups_section}
Considering an oriented abelian group, the group that first comes to mind is the unit circle $\mathbb{S}$. 
Consider the group homomorphism $\phi:\mathbb{R}\rightarrow \mathbb{S}$  given by $\phi(x)=\exp(2\pi i x)$,  which has the kernel $\Z$.
From this homomorphism we can observe a connection between ordered and oriented abelian groups.
In particular, for any oriented abelian group 
$A$, the corresponding ordered abelian group $A^*$
is defined on the set $\Z\times A$, where $A^*$ is equipped with a binary operation $+$ and an order $<$, defined in a manner that makes it an ordered abelian group, without going into the full details here. 	Conversely, for an ordered abelian group $G$ with distinguished positive element $1$, one can construct an oriented abelian group $G_{\mathrm{mod }1}$. The underlying set of $G_{\mathrm{mod }1}$ is defined to be $[0,1):=\{g\in G:0\leq g<1\}$ where the orientation is induced by the order relation on $G$ and addition is defined as 
	\beq \nonumber g+_1 h: = \left\{ \begin{array}{ll}
		g+h & \textrm{ if $g+h<1$ }\\
		g+h-1 & \textrm{otherwise}
	\end{array} \right.\eeq
for $g,h\in [0,1)$. 
Additionally, the regular density of the corresponding ordered and oriented groups are equivalent. Details on how the correspondence between ordered and oriented groups is defined, as well as the relationship between the theories of ordered and oriented structures, can be found in \cite{thesis_ayhan}.

\medskip
 In the setting of ordered abelian groups, the condition that $[p]G=[p]H$ for every prime $p$ is sufficient to ensure that $G$ and $H$ are elementarily equivalent (See \cite{Robinson-Zakon}).
 For oriented abelian groups, however, this criterion is no longer adequate. Since these groups may also contain torsion elements, additional information regarding the torsion is also necessary. Therefore, we use an invariant  $e(A,p)$ for the torsion part of the oriented abelian group $A$, defined as follows.
 
\medskip
	Let $e\in \N$ and let $p$ be any prime. If $e=0$, we define $\operatorname{Tor}_{p^{0}}(A):=\{0\}$.
	For an oriented abelian group $A$, $\operatorname{Tor}_{p^{e}}(A)$ is always a subgroup of $\operatorname{Tor}_{p^{e+1}}(A)$. Indeed, we have the following chain of successive torsion subgroups
    \beq \nonumber\operatorname{Tor}_{p^{0}}(A)\leq \operatorname{Tor}_{p^{1}}(A) \leq \ldots \leq \operatorname{Tor}_{p^{e}}(A) \leq \ldots \leq \operatorname{Tor}_{p^{e+k}}(A)\leq \ldots,\eeq   
where $k\in \N$.  If this chain terminates, then there is some $e\in \N$ such that $\operatorname{Tor}_{p^{e+k}}(A)=\operatorname{Tor}_{p^{e}}(A)$ for all $k\in \N$. In this case, we put $e(A,p)=e$. Otherwise, we put $e(A,p)=\infty$.  

\medskip
The following result gives us a characterization to determine the theory of regularly dense oriented abelian groups.
	\begin{lemma}\label{8.1.9}  \cite[Corollary 8.1.9]{thesis_ayhan}
		Let $A$ and $B$ be regularly dense oriented abelian groups. Then $A$ and $B$ are elementarily equivalent if and only if for every prime $p$, we have $[p]A=[p]B$ and $e(A,p)=e(B,p)$.
	\end{lemma} 

For $i\in \N^{>0}$, let $p_i$ be the $i$-th prime. Define
	$\vec{d}=(d_i)_{i}$ and  $\vec{e}=(e_i)_{i}$,  where $d_i,e_i\in \N\cup \{\infty\}$. Let $T(\vec{d},\vec{e})$ be the $L_{\mathcal{O}}$-theory whose models are $(A,+,-,0,\mathcal{O})$
	such that 
	$A$ is a regularly dense oriented abelian group, and for every prime $p_i$, we have $[p_i] A=p_{i}^{d_i}$ and $e(A,p_i)=e_i$  .

\medskip
Note that regular density is a first-order property in the language $L_{\mathcal{O}}$.
	The following is a direct result of Lemma \ref{8.1.9}.
	\begin{corollary}\label{complete}
	For fixed $\vec{d}$ and $\vec{e}$, the theory	$T(\vec{d},\vec{e})$ is complete.
	\end{corollary}

Now we focus on divisible oriented abelian groups, since quantifier elimination can be obtained in this case.
The following result specifies the possible values of $e(A,p)$
for a divisible oriented abelian group $A$.

	\begin{lemma}\cite[Lemma 8.1.4]{thesis_ayhan}
		Let $A$ be a divisible oriented abelian group. Then, for every prime $p$, $e(A,p)$ is either $\vec{0}$ or $\vec{\infty}$.
	\end{lemma}
    
Here we give some examples of models of $T(\vec{0},\vec{e})$, where $\vec{e}\in \{\vec{0},\vec{\infty}\}$. In the torsion case, both the unit circle $\mathbb{S}$ and the group of complex roots of unity $\mathbb{U}$ are models of $T(\vec{0},\vec{\infty})$. For the torsion-free case, first consider the ordered abelian group $G_a=\Z+a\Q$ for some irrational $a\in \R\setminus \Q$, and let $A_a=(G_a)_{\mathrm{mod }1}$. Since $A_a$ and $\Q$ are isomorphic as groups, $A_a$ is divisible. Now we show that $A_a$ is torsion-free, so let $\overline{aq}\in G_a/\Z$ be such that $n\overline{aq}=\overline{0}$. It follows then $naq\in \Z$ and if $n\neq 0$, we get $aq\in \Q$. Since $a\in \R\setminus \Q$, we conclude that $q=0$. Hence $e(A_a,p)=0$ for all primes $p$, and therefore $A_a$ is a model of $T(\vec{0},\vec{0})$.

\medskip

 We restate the quantifier-elimination result for divisible oriented abelian groups, which will be used in the following section.
 
\begin{theorem}\label{qe_orientedgps}\cite[Lemma 6.3.2]{melissa}
For	$\vec{e}\in \{\vec{0},\vec{\infty}\}$, the $L_\mathcal{O}$-theory $T(\vec{0},\vec{e})$ has quantifier elimination.
\end{theorem}

As mentioned in the Introduction, oriented minimality (\emph{$\mathcal{O}$-minimality}) can be regarded as the counterpart of o-minimality in ordered groups. An oriented abelian group $A$ is called \emph{$\mathcal{O}$-minimal} if every definable subset of $A$ is a finite union of points and orientation intervals, where ‘definable’ refers to sets definable in the structure $(A,+,\mathcal{O})$, possibly with parameters. It is well-known that a densely ordered abelian group $(G,+,<)$ is o-minimal if and only if it is divisible (See \cite{PS}); similarly, by \cite{D-C} an oriented abelian group $A$ is $\mathcal{O}$-minimal exactly when $A\models T(\vec{0},\vec{\infty})$. Moreover, as in the ordered case, it is shown in \cite{melissa} that $\mathcal{O}$-minimality implies NIP.

\medskip
Having established that the theory $T(\vec{0},\vec{\infty})$ has NIP, we note that, in the following subsection, similar properties are also valid for pairs of such models, which serves as a motivating role for our VC-density calculations.

\subsection{Pairs of Divisible Oriented Abelian Groups}\label{The Theory of Pairs of Oriented Abelian Groups}
Recall that  we denote the language of pairs of oriented abelian groups by $L_{\mathcal{O}}^p=L_{\mathcal{O}}\cup \{V\}$, where $V$ is a unary predicate interpreted as the dense subgroup. 

\medskip
Let $T\big((\vec{d_1},\vec{e_1}),(\vec{d_2},\vec{e_2})\big)$ be the $L_{\mathcal{O}}^p$-theory whose models are pairs $(A,G)$ such that $A\models T(\vec{d_1},\vec{e_1})$, $G\models T(\vec{d_2},\vec{e_2})$ where $G$ is a dense subgroup of $A$, and $A/G$ is infinite. 

\medskip
A special $L_{\mathcal{O}}^p$-formula in  $\vec{x} = (x_1,\dots,x_m)$ has the form 
$$\exists \vec{y} \big(V(\vec{y}) \wedge \theta_V(\vec{y})\wedge \phi(\vec{x},\vec{y})\big)$$ 
where $\vec{y} = (y_1,\dots,y_n)$ is a tuple of distinct variables, $\theta_V(\vec{y})$ is the $V$-restriction of a $L_{\mathcal{O}}$-formula $\theta(\vec{y})$, and $\phi(\vec{x},\vec{y})$ is an $L_{\mathcal{O}}$-formula.
 In \cite{melissa} it is shown that the theories of divisible pairs are complete, using the fact that every formula in these theories is equivalent to a special $L_{\mathcal{O}}^p$-formula, together with the elimination results of the quantifiers for theories $T(\vec{0},\vec{\infty})$ and $T(\vec{0},\vec{0})$.
Now, we will show that the divisibility assumption on the structures forming pairs ensures quantifier elimination.

\begin{proposition}\cite[Corollary 7.1.4]{melissa}\label{fact}
  Let $(A',G')\subseteq (A,G)$ be such that $A'$ is a regularly dense pure subgroup of $A$ and $G'$ is dense in $A'$. Suppose that $A'/G'$ is infinite and for every prime $p$, $[p]A=[p]A', \ [p]G=[p]G', \ \operatorname{Tor}(A)=\operatorname{Tor}(A'), \ 
		\operatorname{Tor}(G)=\operatorname{Tor}(G')$.
	 Then we have $(A',G')\preceq (A,G)$.  
\end{proposition}  

Note that, by the above proposition, all substructures we will consider in the forthcoming proofs are pure, due to the divisibility assumptions.

\begin{theorem}\label{qe_pairs_1}
    The $L_{\mathcal{O}}^p$-theory $T\big((\vec{0},\vec{0}),(\vec{0},\vec{0})\big)$ admits quantifier elimination.
\end{theorem}  

\begin{proof}
  Let $(A,G)$ and $(B,H)$ be models of $T\big((\vec{0},\vec{0}),(\vec{0},\vec{0})\big)$, and $(C,D)$ be a common substructure. Suppose that $(A,G)\models \phi(a;\vec{c})$ for some quantifier-free formula $\phi(x;\vec{y})$ where $\vec{c}\in C^{|\vec{y}|}$. We aim to show that there exists $b\in B$ such that $(B,H)\models \phi(b;\vec{c})$. 

  \medskip
 Let $(C',D')$ denote the divisible hull of $(C,D)$. Then there exist embeddings $f_1$ and $f_2$ of $(C',D')$ into $(A,G)$ and $(B,H)$, respectively.
   By Proposition \ref{fact}, these embeddings are elementary. Since  $(C',D')\models \exists x \phi(x;\vec{c})$ and $f_2$ is an elementary embedding, $(B,H)\models \exists x \phi(x;\vec{c})$. Hence there exists $b\in B$ such that $(B,H)\models \phi(b;\vec{c})$. Therefore, quantifier-free types over the common substructures are preserved, completing the proof.
     \end{proof}

In order to deal with torsion groups, we add both the divisible hulls and the torsions to obtain pure extensions, which allows us to establish quantifier elimination results in the remaining theories stated in Theorem \ref{main_thm_qe}.

\begin{theorem}\label{qe_pairs_2}
    The $L_{\mathcal{O}}^p$-theory $T\big((\vec{0},\vec{\infty}),(\vec{0},\vec{\infty})\big)$ admits quantifier elimination.
\end{theorem}

\begin{proof}
    Let $(A,G)$ and $(B,H)$ be models of $T\big((\vec{0},\vec{\infty}),(\vec{0},\vec{\infty})\big)$ with a common substructure $(C,D)$. Let $\phi(x;\vec{y})$ be a quantifier-free formula such that $(A,G)\models \phi(a;\vec{c})$ with parameters $\vec{c}\in C^{|\vec{y}|}$ and $a\in A$.

\medskip
Let $(C',D')$ be the divisible hull of $(C,D)$, then there exist elementary embeddings 
 \begin{align*}
      f_1:\big(C'+\operatorname{Tor}(A),D'+\operatorname{Tor}(G)\big)\longrightarrow (A,G),\\ 
      f_2:\big(C'+\operatorname{Tor}(B),D'+\operatorname{Tor}(H)\big)\longrightarrow (B,H),
 \end{align*}   
where $C'+\operatorname{Tor}(A)$ denotes the smallest group containing $C'$ and $\operatorname{Tor}(A)$.
 
 \medskip
If $a\in \operatorname{Tor}(A)$, then $a\in \operatorname{Tor}_{p^k}(A)$ for some prime $p$ and $k\in \mathbb{Z}$. Note that $G\preceq A$ gives $\operatorname{Tor}(G)=\operatorname{Tor}(A)$, and therefore $a\in \operatorname{Tor}(G)$ as well. Since $|\operatorname{Tor}_{p^k}(A)| = p^k$, we may enumerate its elements as $a_1<_0 \cdots <_0 a_{p^k}$ so that $a = a_j$ for some $j \in \{1,\dots,p^k\}$.
 We also enumerate the elements of $\operatorname{Tor}_{p^k}(B)$ as $b_1<_0\cdots <_0 b_{p^k}$. We then associate $a=a_j$ with $b_j$, thereby defining a  bijection between $\operatorname{Tor}_{p^k}(A)$ and $\operatorname{Tor}_{p^k}(B)$ that preserves the order $<_0$. Repeating this process for all $p^k$-components produces a canonical correspondence between $\operatorname{Tor}(A)$ and $\operatorname{Tor}(B)$.
  Extending this correspondence naturally, we obtain an isomorphism 
$$f: \big(C'+\operatorname{Tor}(A),D'+\operatorname{Tor}(G)\big)\longrightarrow \big(C'+\operatorname{Tor}(B),D'+\operatorname{Tor}(H)\big)$$
whose restriction to $(C',D')$ is the identity. Since $(A,G)\models \phi(a;\vec{c})$ and $f_1$ is elementary, it follows that 
$$\big(C'+\operatorname{Tor}(A),D'+\operatorname{Tor}(G)\big)\models \exists x \phi(x;\vec{c}).$$
Applying the isomorphism $f$, we obtain $$\big(C'+\operatorname{Tor}(B),D'+\operatorname{Tor}(H)\big)\models \exists x \phi(x;\vec{c}).$$
As $f_2$ is also elementary, we conclude that $(B,H)\models \exists x \phi(x;\vec{c})$. Thus, 
there exists $b\in B$ such that $(B,H)\models \phi(b;\vec{c})$. This establishes that quantifier-free types over common substructures are preserved, which completes the proof.
\end{proof}

\begin{theorem}\label{qe_pairs_3}  
The $L_{\mathcal{O}}^p$-theory $T\big((\vec{0},\vec{\infty}),(\vec{0},\vec{0})\big)$ admits quantifier elimination.
\end{theorem}

\begin{proof}
  The proof proceeds as in the previous case, with the only modification that no torsion part is added to the distinguished subgroup, since it is torsion-free in models of $T\big((\vec{0},\vec{\infty}),(\vec{0},\vec{0})\big)$.
\end{proof}

\section{VC-Density Calculations}\label{main_result_section}
Let $L$ be a first order language and $\psi(\vec{x};\vec{y})$ be a \textit{partitioned $L$-formula} in object variables $\vec{x}=(x_1,\ldots,x_m)$ and parameter variables $\vec{y}=(y_1,\ldots,y_n)$. The \textit{shatter function}, \textit{VC-dimension}, and \textit{VC-density} associated to $\psi$ are elementary invariants and our original theories are complete, we will simply use the notations $\pi_\psi$, $\VC(\psi)$, and $\vc(\psi)$ to denote these combinatorial parameters, respectively. 

\medskip
In \cite{VC-density-I}, Aschenbrenner et al. introduce general methods yielding uniform bounds on VC-density of formulas. One of them, which will be a key tool for our main theorems, is stated below as Theorem \ref{keytool}. Since we are interested in theories with quantifier elimination, instead of counting the diversity of types over finite sets, we will benefit from the relation between \textit{breadth} and UDTFS. A collection $\mathcal{S}$ of sets is said to have \textit{breadth $d$} if any nonempty finite intersection of members of $\mathcal{S}$ is given by the intersection of $d$ of them. The following result shows that finite breadth gives uniform definability of types over finite sets.

\begin{proposition}\cite[Lemma 5.2]{VC-density-I}\label{breadth}
     Let $\mathcal{M}$ be an $L$-structure and $\Delta(x;\vec{y})$ be a finite set of partitioned $L$-formulas in the single object variable $x$. If $\mathcal{S}_{\Delta}:=\{\psi(M;\vec{c}): \psi\in \Delta, \vec{c}\in M^{|\vec{y}|}\}$ has breadth $d$, then $\Delta$ has UDTFS with $d$ parameters.
\end{proposition}

\medskip\noindent
A complete theory $T$ has the \textit{VC $d$ property} if any finite set of partitioned $L$-formulas in the single object variable has UDTFS with $d$ parameters.  

\begin{theorem}\cite[Corollary 5.6]{VC-density-I}\label{keytool}
Let $\Omega$ be a collection of partitioned $L$-formulas in the single object variable $x$. Suppose that any partitioned $L$-formula is equivalent in $T$ to a boolean combination of formulas from $\Omega$. If any finite subset of $\Omega$ has UDTFS in $T$ with $d$ parameters, then $T$ has the VC $d$ property.
\end{theorem}

\medskip
We now state the following corollary, which serves as the final step before the proofs of the main theorems.
 
\begin{corollary}\cite[Corollary 5.8]{VC-density-I}\label{vc_d_vcdensity}
If $T$ has the VC $d$ property, then any partitioned formula $\psi(\vec{x};\vec{y})$ has the VC-density at most $d|\vec{y}|$.
\end{corollary} 

\medskip
We are now ready to prove Theorem \ref{main_thm1}  stated in the Introduction. 

\begin{theorem}\label{mainthm_1}
The $L_\mathcal{O}$-theories $T(\vec{0},\vec{e})$ where $\vec{e}\in \{\vec{0},\vec{\infty}\}$ have the VC $1$ property.
\end{theorem}

\begin{proof}
The proof proceeds exactly in the same way for both of the theories $T(\vec{0},\vec{0})$ and $T(\vec{0},\vec{\infty})$, so we only give the proof for $T(\vec{0},\vec{0})$.

\medskip\noindent
Let $A$ be a model of $T(\vec{0},\vec{0})$ and consider the collection $\Omega$ of the $L_{\mathcal{O}}$-formulas \beq x=\vec{k}\vec{y}\ \text{and}\ \mathcal{O}(0,x,\vec{k}\vec{y}),\label{atomic1}\eeq
in a single object variable $x$ and a tuple of parameter variables $\vec{y}$, where the lengths of the tuples $\vec{y}$ vary, and $\vec{k}$ ranges over  $\mathbb{Z}^{|\vec{y}|}$.

\medskip\noindent
 By Theorem~\ref{qe_orientedgps} and Theorem~\ref{keytool}, it remains to show that any finite $\Delta\subseteq \Omega$ has UDTFS with $1$ parameter. We show this by using Proposition \ref{breadth}, so consider the collection $\mathcal{S}_\Delta$. Observe that any set defined by a formula of the form (\ref{atomic1}) is either a singleton or an orientation interval in $A$. Take $S_1,S_2\in \mathcal{S}_{\Delta}$ such that $S_1\cap S_2\neq \emptyset$. If $S_1$ or $S_2$ is a singleton, then we are done. We may assume then $S_1$ and $S_2$ are both orientation intervals in $A$, so let $S_1=(0,a_1)_{\mathcal{O}}$ and $S_2=(0,a_2)_{\mathcal{O}}$, where $a_1,a_2\in A$. As their intersection is nonempty, we have necessarily $S_1\subseteq S_2$ or $S_2\subseteq S_1$. It follows then $\mathcal{S}_\Delta$ has breadth $1$.
\end{proof}

The following result follows from Corollary \ref{vc_d_vcdensity}.

\begin{corollary}\label{vcbound-L_O}
Any partitioned $L_{\mathcal{O}}$-formula $\psi(\vec{x};\vec{y})$ has the VC-density at most $|\vec{y}|$ in $T(\vec{0},\vec{e})$, where $\vec{e}\in \{\vec{0},\vec{\infty}\}$.   
\end{corollary}

\begin{corollary}\label{vcbound-L_O}
For $\vec{e}\in \{\vec{0},\vec{\infty}\}$, the $L_\mathcal{O}$-theory $T(\vec{0},\vec{e})$ is dp-minimal.   
\end{corollary}

\medskip
Finally, we prove Theorem \ref{main_thm2} given in the Introduction.
\begin{theorem}\label{mainthm_2}
    The $L_{\mathcal{O}}^p$-theories $T\big((\vec{0},\vec{0}),(\vec{0},\vec{0})\big)$ and $T\big((\vec{0},\vec{\infty}),(\vec{0},\vec{e})\big)$, where $\vec{e}\in \{\vec{0},\vec{\infty}\}$ have the VC $2$ property.
\end{theorem}
\begin{proof}
As the same arguments apply to all the theories, we prove the statement only for $T\big((\vec{0},\vec{0}),(\vec{0},\vec{0})\big)$, so let $(A,G)\models T\big((\vec{0},\vec{0}),(\vec{0},\vec{0})\big)$. 

\medskip
Consider the set $\Omega$ containing the following $L_{\mathcal{O}}^p$-formulas in a single variable $x$ and a tuple $\vec{y}$ of various lengths:  
\begin{itemize}
\item[i.] $x=\vec{k}\vec{y}$,
    \item[ii.] $\mathcal{O}(0,x,\vec{k}\vec{y})$,
    \item[iii.] $V(x+\vec{k}\vec{y})$, 
\end{itemize}
where $\vec{k}$ ranges in $\mathbb{Z}^{|\vec{y}|}$.

\medskip
Since the theory $T\big((\vec{0},\vec{0}),(\vec{0},\vec{0})\big)$ eliminates quantifiers by Theorem \ref{qe_pairs_1}, it suffices by Theorem \ref{keytool} to show that any finite $\Delta\subseteq \Omega$ has UDTFS with $2$ parameters. By Proposition \ref{breadth}, we only show that for any finite $\Delta\subseteq \Omega$, the collection $\mathcal{S}_\Delta$ has breadth $2$. Take $B_1,\ldots,B_n\in \mathcal{S}_\Delta$ such that $\bigcap_{i=1}^n B_i\neq \emptyset$ and $n\geq 2$. 

\medskip
Observe first $B_i$ is either a singleton, or an orientation interval, or a coset of $G$ depending on whether it is defined by a formula of the form $(i), (ii)$, or $(iii)$. If one of the $B_i$ is a singleton, say $B_j$, then $\bigcap_{i=1}^n B_i=B_j$ as the intersection on the left is nonempty. Thus we may assume that there is no $B_i$ defined by a formula of the form $(i)$. 

\medskip
  If there are $B_i$ defined by $(ii)$, then their intersection, as it is nonempty, is given by the one whose supremum is the least one. Denote this intersection by $B_{i_2}$, with $i_2\in \{1,\ldots,n\}$. Note that any two cosets of $G$ intersect nontrivially if and only if they are the same coset. Thus, if there are $B_i$ defined by $(iii)$, their intersection can be given by any of them. Let $B_{i_3}$ denote this intersection, with $i_3\in \{1,\ldots,n\}$. Therefore, we conclude that $\bigcap_{i=1}^n B_i$ is equal to either $B_{i_2}$, or $B_{i_3}$, or $B_{i_2}\cap B_{i_3}$, depending on whether the sets $B_i$ defined by $(ii)$ or $(iii)$ occur among $B_1,\ldots,B_n$. Hence $\mathcal{S}_\Delta$ has breadth $2$.
\end{proof}

 \begin{corollary}\label{vcbound-L_O^p}
In the $L_{\mathcal{O}}^p$-theories $T\big((\vec{0},\vec{0}),(\vec{0},\vec{0})\big)$ and
$T\big((\vec{0},\vec{\infty}),(\vec{0},\vec{e})\big)$, where $\vec{e}\in \{\vec{0},\vec{\infty}\}$, any partitioned formula $\psi(\vec{x};\vec{y})$ has the VC-density at most $2|\vec{y}|$.
 \end{corollary}

Given a formula $\psi$, one can easily see that if the VC-density of $\psi$ is finite, then so is its VC-dimension (see \cite{VC-density-I}). Using the Corollary above and the connection of VC-dimension with NIP theories, we derive the following result.

\begin{corollary}
     The $L_{\mathcal{O}}^p$-theory $T\big((\vec{0},\vec{0}),(\vec{0},\vec{0})\big)$ has NIP.
\end{corollary}

\subsection{\boldmath Sharpness Of The Bounds And Consequences}

Note first that in any model of a complete theory without finite models, one can always find a formula whose VC-density is exactly the size of its parameter tuple. Indeed, the formula $\bigvee_{i=1}^{n} x=y_i$ witnesses this fact. Therefore, the VC $1$ property yields the best possible uniform bound on VC-densities of formulas in terms of the length of the parameter variables. Consequently, the bound established in Corollary \ref{vcbound-L_O} is optimal, whereas the bound obtained in Corollary \ref{vcbound-L_O^p} need not be. We start with proving that it is also sharp.


\medskip
In order to prove our claim, we introduce the following partitioned $L_{\mathcal{O}}^p$-formula in the single object variable $x$ and single parameter variable $y$:
$$\sigma(x;y): \mathcal{O}(y,x,2y) \lor V(x-y).$$

\begin{proposition}\label{optimality}
    In any model of any $L_{\mathcal{O}}^p$-theory appearing in Theorem \ref{mainthm_2}, the $L_{\mathcal{O}}^p$-formula $\sigma(x;y)$ has the VC-density $2$.
\end{proposition}

\begin{proof}
Our proof will focus on a model $(A,G)$ of $ T\big((\vec{0},\vec{0}),(\vec{0},\vec{0})\big)$ but the proof is valid for any $L_{\mathcal{O}}^p$-theory in Theorem \ref{mainthm_2}. 

\medskip\noindent
We show that for any $k\geq 2$, there is a $k$-element subset of $A$ whose subsets of size $\leq 2$ are all cut out by $\sigma$. For a fixed $k\geq 2$, take $B=\{x_1,\ldots,x_k\}\subseteq A$ such that $\mathcal{O}(0,x_i,x_j)$ for any $i<j$ and $x_i-x_j\notin G$ for any $i\neq j$. Suppose also that $$x_i<_0 2x_i<_0 x_{i+1}$$
for each $i=1,\ldots,k-1$. We have clearly then $\{x_{i}\}=\sigma(A;x_i)\cap B$. 

\medskip\noindent
By using the density of $G$, for any $i\neq j$ there exists $\alpha_i\in A$ such that $x_j\in \alpha_i +G$ and
$$x_{i-1}<_0\alpha_i <_0x_i<_02\alpha_i <_0x_{i+1}.$$
By assumption on $x_i$, we get that
\beq \nonumber 
\{x_i,x_j\}=\sigma(A;\alpha_i)\cap B.
\eeq 

\medskip\noindent
Hence $\sigma$ cuts out any subset of $B$ of cardinality $\leq 2$. It follows then 
\beq \nonumber
\pi_{\sigma}(k)\geq \binom{k}{0}+\binom{k}{1}+\binom{k}{2}.
\eeq
Combining this inequality with Corollary \ref{vcbound-L_O^p}, we get the following desired result
\beq \nonumber 
\vc(\sigma)= 2>|y|=1.
\eeq 
\end{proof}

\begin{proposition}\label{non_dpminimality-L_O^p}
None of the $L_{\mathcal{O}}^p$-theories in Theorem \ref{mainthm_2} is dp-minimal.
\end{proposition}
\begin{proof}
 Let $T_{\mathcal{O}}^p$ be any of the theories in Theorem \ref{mainthm_2} and $(A,G)$ be a model of $T_{\mathcal{O}}^p$. Let $\phi(x;y)$ and $\psi(x;y)$ be the formulas $\mathcal{O}(y,x,2y)$ and $V(x-y)$, respectively. We find two sequences $(\alpha_i)_{i\in\N}$ and $(\beta_i)_{i\in\N}$ in $A$ such that for some $i,j\in \N$, the following set of $L_{\mathcal{O}}^p$-formulas: 
 \[
    \{\phi(x;\alpha_i),\psi(x;\beta_j)\}\cup\{\lnot \phi(x;\alpha_k):k\neq i\}\cup\{\lnot \psi(x;\beta_l):l\neq j\},
  \]
is realized in $(A,G)$.

\medskip
Let $(\alpha_i)_{i\in\N}$ be such that $\mathcal{O}(0,x_i,x_j)$ for any $i<j$ and $\alpha_i<_0 \alpha_{i+1}<_0 2\alpha_i$ for each $i\in\N$. Assume also that $\alpha_i-\alpha_j\notin G$ for $i\neq j$. Take $(\beta_i)_{i\in\N}$ with $\beta_i=\alpha_i$ for all $i\in\N$. Fixing $i,j\in\N$ and applying the density of $G$ in $A$, we find some 
\beq \nonumber c\in(\alpha_i,2\alpha_i)_\mathcal{O}\cap \alpha_j+G.\eeq By the construction of the sequences, $c$ is not contained in \beq \nonumber \bigcup_{k\neq i,l\neq j}(\alpha_k,2\alpha_k)_\mathcal{O}\cup \alpha_l+G.\eeq
The proof holds in any model of any $L_{\mathcal{O}}^p$-theories, hence we are done.  
\end{proof}

Although dp-minimality fails in the case of pairs, we are able to prove finite dp-rank in each of the corresponding theories. In \cite{VC-density-II}, it is shown that if a theory has the dp-rank $\geq n$, then there is a formula $\psi$ in a single parameter variable such that $\vc(\psi)\geq n$. Combining this with Corollary \ref{vcbound-L_O^p} and Proposition \ref{optimality}, we get the following result.

\begin{corollary}
Given $\vec{e}\in \{\vec{0},\vec{\infty}\}$,  each of the $L_{\mathcal{O}}^p$-theories $T\big((\vec{0},\vec{\infty}),(\vec{0},\vec{e})\big)$ and $T\big((\vec{0},\vec{0}),(\vec{0},\vec{0})\big)$ has the dp-rank $2$.
\end{corollary}

\section*{Acknowledgements}
We would like to thank Ayhan Günaydın for providing insightful comments on an earlier draft, which significantly improved our results.


\end{document}